\documentclass[11pt, a4paper]{amsart}

\usepackage{amsfonts,amsmath,amssymb, amscd,fullpage}
\usepackage{stmaryrd}
\usepackage[all]{xy}

\newtheorem{theorem}{Theorem}[section]
\newtheorem{lemma}[theorem]{Lemma}
\newtheorem{definition}[theorem]{Definition}
\newtheorem{proposition}[theorem]{Proposition}
\newtheorem{corollary}[theorem]{Corollary}

\newtheorem{remark}[theorem]{Remark}

\theoremstyle{definition}

\newcommand\pf{\begin{proof}}
\newcommand\epf{\end{proof}}

\newcommand\C{\mathbb{C}}
\newcommand\cs{\mathrm{C}}

\newcommand\rep{\mathcal R}

\renewcommand{\Im}{\image}
\DeclareMathOperator{\image}{Im}

\numberwithin{equation}{section}

\title{Quantum subgroups of the compact quantum group $SU_{-1}(3)$}

\author{Julien Bichon and Robert Yuncken}
\address{
Laboratoire de Math\'ematiques,
Universit\'e Blaise Pascal,
Complexe universitaire des C\'ezeaux,
63177~Aubi\`ere Cedex, France}
\email{Julien.Bichon@math.univ-bpclermont.fr - Robert.Yuncken@math.univ-bpclermont.fr}


\begin{document}

\begin{abstract}
 We study the (compact) quantum subgroups of the compact quantum group $SU_{-1}(3)$: we show that any non-classical such quantum subgroup is a twist of a compact subgroup of $SU(3)$ or is isomorphic to a quantum subgroup of $U_{-1}(2)$.
\end{abstract}

\maketitle

\section{introduction}

Quantum groups, named after Drinfeld's seminal work \cite{dr}, are natural Hopf algebraic generalizations of usual groups, arising in several branches of mathematics. As in classical group theory, the problem of their classification is a fundamental one.

An important aspect of the classification problem for quantum groups is the determination of the quantum subgroups of the known quantum groups. Let us recall some significant contributions to this topic.

\begin{enumerate}
 \item Podles \cite{po} was the first to consider this problem, and he classified the compact quantum subgroups of Woronowicz' quantum group $SU_{q}(2)$, for  $q \in [-1,1] \setminus\{0\}$. For other approaches, see \cite{bn} (for the finite quantum subgroups when $q=-1$) or \cite{fst} (when $q \not=-1$).
\item The finite quantum subgroups of $GL_q(n)$ were classified by M\"uller \cite{mu}, for $q$ an odd root of unity. From this work arose in particular an infinite family of pairwise non-isomorphic Hopf algebras of the same dimension: this was one of the series of counterexamples to Kaplansky's tenth conjecture. 
\item The work of M\"uller was subsequently generalized by Andruskiewitsch and Garcia in \cite{ag}, where they determined the quantum subgroups of $G_q$, with $G$ a connected, simply connected simple algebraic group and $q$ a root of unity of odd order. 
\item Another generalization of Muller's work was provided by Garcia \cite{gar}, who studied the two-parameter deformations $GL_{\alpha, \beta}(n)$, and classified the quantum subgroups in the odd root of unity case.
\item The compact quantum subgroups of $SO_{-1}(3)$ were determined by Banica and the first author in \cite{bb09}: these are the compact quantum groups acting faithfully on the classical space consisting of $4$ points.
\item The compact quantum subgroups of $O_n^*$, the half-liberated orthogonal quantum groups from \cite{basp}, were determined by Dubois-Violette and the first author in \cite{bicdub}. 
\end{enumerate}

From these works emerged several new interesting classes of quantum groups, and several hints of what the classification of quantum groups should be. The approaches in (2), (3) and (4) deal with non-semisimple quantum groups and do not treat the case $q=-1$, while this is certainly the most interesting case if we have semisimple finite quantum groups in mind.
The present paper is a contribution to the case $q=-1$: we determine the compact quantum subgroups of the compact quantum group $SU_{-1}(3)$, as follows.

\begin{theorem}\label{subSu3}
 Let $G$ be a non-classical compact quantum subgroup of $SU_{-1}(3)$. Then one of the following statements holds.
\begin{enumerate}
 \item  $G$ is isomorphic to a $K_{-1}$, a twist at $-1$ of a compact subgroup $K \subset SU(3)$ containing the subgroup of diagonal matrices having $\pm 1$ as entries.
\item  $G$ is isomorphic to a quantum subgroup of $U_{-1}(2)$. 
\end{enumerate}
\end{theorem}

The quantum subgroups of $U_{-1}(2)$ can be determined by using similar techniques to those of Podles \cite{po}.
We shall not discuss this here. Note  that it follows from Theorem \ref{subSu3} and its proof that if $G$ is a non-classical compact quantum subgroup of $SU_{-1}(3)$ acting irreducibly on $\C^3$,
then  $G$ is isomorphic to a $K_{-1}$, a twist at $-1$ of a compact subgroup $K \subset SU(3)$ containing the subgroup of diagonal matrices having $\pm 1$ as entries, and acting irreducibly on $\C^3$. Thus for any quantum subgroup of $SU_{-1}(3)$ acting irreducibly on the fundamental representation, the tensor category of representations is  symmetric (in Hopf algebra terms, the Hopf algebra $\rep(G)$ is cotriangular). This seems to be an interesting phenomenon, that does not hold in general (e.g. for the quantum group $U_{-1}(2)$).

As in \cite{bb09}, the starting point is that $SU_{-1}(3)$ is a twist at $-1$ of the classical group $SU(3)$ (a 2-cocycle deformation).
This furnishes a number of representation-theoretic tools, developed in Section 3, to study the $\cs^*$-algebra $C(SU_{-1}(3))$ and its quotients, which are used in an essential way to prove Theorem \ref{subSu3}. Note that the representation theory of twisted function algebras on finite groups is fully discussed in \cite{eg2}, with a precise description of the irreducible representations. However the fusion rules, which would lead to the full classification of the Hopf algebra quotients, are not discussed in \cite{eg2}, and we do not see any general method to compute them. What we get here in the case of $SU_{-1}(3)$ are some partial fusion rules, for some special representations of $C(SU_{-1}(3))$, which however are sufficiently generic to get the necessary information to classify the quantum subgroups. 

The paper is organized as follows. Section 2 consists of preliminaries. In Section 3 we recall the twisting (2-cocycle deformation) procedure for Hopf algebras and develop the aforementioned representation-theoretic tools for representations of twisted $\cs^*$-algebras of functions. In Section 4 we briefly recall how the quantum group $SU_{-1}(2m+1)$ can be obtained by twisting, and Section 5 is devoted to the proof of Theorem \ref{subSu3}.

\bigskip

We would like to thank S.~Echterhoff for informative discussions.

\section{Preliminaries}

\subsection{Compact quantum groups} We first recall some basic facts concerning compact quantum groups. The book \cite{ks}
is a convenient reference for the topic of compact quantum groups,
and all the defintions we omit can be found there.
All algebras in this paper will be unital, and 
$\otimes$ will denote the minimal tensor product of $\cs^*$-algebras
as well as the algebraic tensor product; this should cause no confusion.

\begin{definition}
A \textbf{Woronowicz algebra} is a $\cs^*$-algebra $A$ endowed with  
a $*$-morphism 
$\Delta : A \to A \otimes A$ 
satisfying the coassociativity condition and the cancellation law
$$\overline{\Delta(A)(A \otimes 1)} = A \otimes A 
= \overline{\Delta(A)(1\otimes A)}$$
The morphism $\Delta$ is called
the comultiplication of $A$.
\end{definition}

The category of 
Woronowicz algebras is defined in the obvious way (see \cite{wa0} for details). 
A commutative Woronowicz algebra is necessarily isomorphic with $C(G)$, the
 algebra of continuous functions on a compact group $G$, unique up to isomorphism,  
and the category of \textbf{compact quantum groups} is defined to be the category
dual to the category of Woronowicz algebras.
Hence to any Woronowicz algebra $A$ corresponds a unique compact quantum group
according to the heuristic formula $A = C(G)$. 

Woronowicz's original definition for matrix compact quantum groups
\cite{wo} is still the most useful in concrete situations, and 
we have the following fundamental result \cite{wo2}.

\begin{theorem}
 Let $A$ be a $\cs^*$-algebra endowed with a $*$-morphism
 $\Delta : A \to A \otimes A$. Then $A$ is a Woronowicz algebra if and only if
there exists a family of unitary matrices
$(u^\lambda)_{\lambda \in \Lambda} \in M_{d_\lambda}(A)$ satisfying the following three 
conditions.

\begin{enumerate}
\item The $*$-subalgebra $A_0$ generated by the entries $u_{ij}^\lambda$
of the matrices $(u^\lambda)_{\lambda \in \Lambda}$ is dense in $A$.

\item For $\lambda \in \Lambda$ and $i,j \in \{1, \ldots, d_\lambda \}$,
one has $\Delta(u_{ij}^\lambda) = \sum_{k=1}^{d_{\lambda}}
u_{ik}^\lambda \otimes u_{kj}^\lambda$.

\item For $\lambda \in \Lambda$, the transpose matrix $(u^\lambda)^t$ is
invertible.
\end{enumerate}
\end{theorem}

In fact the $*$-algebra $A_0$ in the theorem is canonically defined, and is what we call a compact Hopf algebra (a CQG algebra in \cite{ks}): a Hopf $*$-algebra
having all its finite-dimensional comodules equivalent to unitary ones, or equivalently a  Hopf $*$-algebra having a positive and faithful Haar integral (see \cite{ks} for details).
The counit and antipode of $A_0$, denoted respectively $\varepsilon$ and $S$, 
are referred to as the counit and antipode of $A$.
The Hopf algebra $A_0$ is called the \textbf{algebra of representative functions}
on the compact quantum group $G$ dual to $A$, with another heuristic formula
$A_0 = \rep(G)$.

Conversely, starting from
a compact Hopf algebra, the universal
$\cs^*$-completion yields a Woronowicz algebra in the above sense: see the book 
\cite{ks}. In fact, in general, there are possibly several different $\cs^*$-norms
on $A_0$, in particular the reduced one (obtained from the GNS-construction associated to the Haar integral), but we will not be concerned with this problem, the compact quantum groups considered in this paper being co-amenable.

Of course, any group-theoretic statement about a compact quantum group $G$ must be
interpreted in terms of the Woronowicz  algebra $C(G)$ or of the Hopf $*$-algebra $\mathcal R(G)$. In particular, 
as usual, a (compact) \textbf{quantum subgroup} $H \subset G$ corresponds to a surjective Woronowicz algebra
morphism $C(G)\to C(H)$, or to a surjective
Hopf $*$-algebra morphism $\rep(G)\to\rep(H)$.

\subsection{The quantum groups $U_{-1}(n)$ and $SU_{-1}(n)$} In this subsection we briefly recall the definition of the compact quantum groups $U_{-1}(n)$ and $SU_{-1}(n)$ \cite{wo88,koe,ro}.

\begin{definition}
The $*$-algebra $\rep(U_{-1}(n))$ is the universal $*$-algebra generated by variables $(u_{ij})_{1\leq i,j \leq n}$ with relations making the matrix $u=(u_{ij})$ unitary and 
$$u_{ij}u_{kl} = (-1)^{\delta_{ik} + \delta_{jl}} u_{kl}u_{ij}, \ \forall i,j,k,l$$
The $\cs^*$-algebra $C(U_{-1}(n))$ is the enveloping $\cs^*$-algebra of $\rep(U_{-1}(n))$.
\end{definition}

The relations $u_{ij}^*u_{kl} = (-1)^{\delta_{ik} + \delta_{jl}} u_{kl}u_{ij}^*$ automatically hold in $\rep(U_{-1}(n))$
and  $C(U_{-1}(n))$, hence the matrix $u^t$ is also unitary.
It follows that  $\rep(U_{-1}(n))$ is a compact Hopf $*$-algebra, and hence that $C(U_{-1}(n))$ is a Woronowicz algebra, with comultiplication, counit and antipode defined by 
$$\Delta(u_{ij}) = \sum_k u_{ik}\otimes u_{kj}, \ 
\varepsilon(u_{ij}) = \delta_{ij}, \ S(u_{ij}) = u_{ji}^*$$
The quantum determinant
$$D= \sum_{\sigma \in S_n} u_{1\sigma(1)} \cdots u_{n\sigma(n)} =\sum_{\sigma \in S_n} u_{\sigma(1)1} \cdots u_{\sigma(n)n} $$
is a unitary central group-like element of $\rep(U_{-1}(n))$. 

\begin{definition}
 The $*$-algebra  $\rep(SU_{-1}(n))$ is the quotient of  $\rep(U_{-1}(n))$ by the $*$-ideal generated by $D-1$, and the $\cs^*$-algebra $C(SU_{-1}(n))$ is the enveloping $\cs^*$-algebra of $\rep(SU_{-1}(n))$.
\end{definition}

It follows, since $D$ is group-like, that  $\rep(SU_{-1}(n))$ is a compact Hopf $*$-algebra, and  that $\cs(SU_{-1}(n))$ is a Woronowicz algebra, with comultiplication, counit and antipode defined by the same formulas as above.

The following Lemma will be used in Section 5.

\begin{lemma}\label{reduc}
 For any $i\in \{1, \ldots , n+1\}$, there exists a surjective Hopf $*$-algebra map $\pi_i :  \rep(SU_{-1}(n+1))\rightarrow \rep(U_{-1}(n))$ whose kernel is the Hopf $*$-ideal generated by the elements $u_{ki}$, $u_{ik}$, $k \not=i$. In particular, if $\pi : \rep(SU_{-1}(n+1))\twoheadrightarrow A$ is a surjective Hopf $*$-algebra  map such that
for some fixed $i$ we have $\pi(u_{ki})=0=\pi(u_{ik})$ for  $k \not=i$, then there exists a surjective Hopf $*$-algebra map $\rep(U_{-1}(n)) \twoheadrightarrow A$.
\end{lemma}

\begin{proof}
 It follows from the definitions that there exists a Hopf $*$-algebra map $\pi_i$  such that $\pi_i(u_{ki})=0=\pi_i(u_{ik})$ for $k \not=i$, $\pi_i(u_{ii})=D^{-1}$, $\pi_i(u_{jk})=u_{jk}$ for $j,k<i$, $\pi_i(u_{jk}) =u_{j,k-1}$ for $j<i$ and $k>i$, $\pi_i(u_{jk})=u_{j-1,k}$ for $j>i$ and $k<i$, $\pi_i(u_{jk})=u_{j-1,k-1}$ for $j,k>i$. By definition $\pi_i$ vanishes on $I$, the $*$-ideal generated by the elements in the statement of the lemma, so induces a surjective $*$-algebra map $\overline{\pi_i} : \rep(SU_{-1}(n+1))/I\rightarrow \rep(U_{-1}(n))$, and it is not difficult to construct an inverse isomorphism to $\overline{\pi_i}$, and hence $I = {\rm Ker}(\pi_i)$. The last assertion is an immediate consequence of the first one.
\end{proof}

\subsection{Representations of $\cs^*$-algebras}
In this short subsection, we collect a few useful facts on representations of $*$-algebras and $\cs^*$-algebras. 
If $A$ is $*$-algebra, a representation of $A$ always means a Hilbert space representation of $A$, i.e. a $*$-algebra map $A \rightarrow \mathcal B(H)$ into the $*$-algebra of bounded operators on a Hilbert space $H$. As usual, the set of isomorphism classes of irreducible representations of $A$ is denoted by $\widehat{A}$. If $\rho, \pi$ are representations of $A$, we write $ \rho \prec \pi$ if $\rho$ is isomorphic to a sub-representation of $\pi$. 

The following classical result will be a key tool. See e.g. \cite{dix} for a proof.

\begin{theorem}\label{extend}
 Let $A \subset B$ be an inclusion of  $\cs^*$-algebras, and let $\rho$ be an irreducible representation of $A$. Then there exists an irreducible representation $\pi$ of $B$ such that $\rho \prec \pi_{|A}$.
\end{theorem}

Let $A$ be a $*$-algebra. If $\rho : A \rightarrow \mathcal B(H)$ is a finite-dimensional representation, then the character of $\rho$ is the linear map $\chi = {\rm tr} \rho$, where ${\rm tr}$ is the usual trace. Two finite-dimensional representations of $A$ are isomorphic if and only if they have the same character. 

\medskip

Now assume that $A$ is a Hopf $*$-algebra. The trivial representation is $\varepsilon$, the counit of $A$. Let $\rho : A \rightarrow \mathcal B(H)$ be a finite-dimensional representation of $A$. Recall that the dual representation $\rho^\vee: A \rightarrow \mathcal B(\overline{H})$ (where $\overline{H}$ is the conjugate Hilbert space of $H$) is defined by $\rho^\vee(a)(\overline{x})=\overline{\rho (S(a^*))(x)}$, for any $a\in A$ and $x \in H$.  We have $\varepsilon \prec \rho\otimes \rho^\vee$, and when $\rho$ is irreducible, this property characterizes $\rho^\vee$ up to isomorphism.

\section{2-cocycle deformations}

We now recall the usual twisting (2-cocycle deformation) construction for Hopf algebras, which is dual to the theory initiated by Drinfeld, and developed by Doi \cite{do}. We also develop the representation theoretic machinery needed to study the quotients of a twisting of a Hopf algebra of representative functions on a compact group.

Let $Q$ be a Hopf $*$-algebra. We use Sweedler's notation
$\Delta(x) = x_{1} \otimes x_{2}$. Recall (see e.g. \cite{do})
that a \textbf{unitary 2-cocycle} on $Q$ is a convolution invertible linear map
$\sigma : Q \otimes Q \longrightarrow \C$ satisfying
$$\sigma(x_{1},y_{1}) \sigma(x_{2}y_{2},z) =
\sigma(y_{1},z_{1}) \sigma(x,y_{2} z_{2})$$
$$\sigma^{-1}(x,y)=\overline{\sigma(S(x)^*,S(y)^*)}$$
and $\sigma(x,1) = \sigma(1,x) = \varepsilon(x)$, for $x,y,z \in Q$.

Following \cite{do} and \cite{sc1}, we associate various $*$-algebras to
a unitary 2-cocycle. 

$\bullet$ First consider the $*$-algebra 
$_{\sigma} \! Q$. As a vector space we have $_{\sigma} \! Q = Q$ and the product and involution 
of $_{\sigma}Q$ are defined to be
$$\{x\}  \{y\} = \sigma(x_{1}, y_{1}) \{x_{2} y_{2}\}, 
 \quad
\{x\}^*=\sigma^{-1}(x_2^*, S(x_1)^*)\{x_3^*\}, \quad x,y \in Q,$$
where an element $x \in Q$ is denoted $\{x\}$, when viewed as an element 
of $_{\sigma} \!Q$.

$\bullet$ We also have the $*$-algebra $Q_{\sigma^{-1}}$. As a vector space we have
$Q_{\sigma^{-1}} = Q$ and the product and involution of  
$Q_{\sigma^{-1}}$ are defined to be
$$\langle x \rangle \langle y \rangle = \sigma^{-1}(x_{2}, y_{2}) \langle x_{1} y_{1} \rangle, \quad 
\langle x\rangle^*=\sigma(S(x_3)^*, x_2^*)\langle x_1^*\rangle,
\quad x,y \in Q.$$
where an element $x \in Q$ is denoted $\langle x \rangle$, when viewed as an element 
of $Q_{\sigma^{-1}}$. The unitary cocycle condition ensures that $_{\sigma} \! Q$
and $Q_{\sigma^{-1}}$ are associative $*$-algebras with $1$ as a unit. The algebras $_{\sigma} \! Q$
and $Q_{\sigma^{-1}}$ are in fact anti-isomorphic, see e.g. \cite{bi03}.

If $Q$ is a compact Hopf algebra, then the Haar integral on $Q$, viewed as a linear map on  $_{\sigma}Q$ and $Q_{\sigma^{-1}}$, is still a faithful state (this can been seen by using the orthogonality relations \cite{wo, ks}).
We denote by $\cs^*_r(_{\sigma}Q)$ and $\cs_r^*(Q_{\sigma^{-1}})$ the respective $\cs^*$-completions obtained from the GNS-constructions associated to the Haar integral.

$\bullet$ Finally we have the 
Hopf $*$-algebra $Q^{\sigma} = {_{\sigma} \! Q}\!_{\sigma^{-1}}$.  
As a coalgebra $Q^{\sigma} 
= Q$. The product and involution of $Q^{\sigma}$ are defined to be
$$[x] [y]= \sigma(x_{1}, y_{1})
\sigma^{-1}(x_{3}, y_{3}) [x_{2} y_{2}], \quad [x]^*=\sigma(S(x_5)^*, x_4^*) \sigma^{-1}(x_2^*, S(x_1)^*) [x_3^*]
\quad x,y \in Q,$$
where an element $x \in Q$ is denoted $[x]$, when viewed as an element 
of $Q^{\sigma}$, 
and we have the following formula for the antipode of 
$Q^{\sigma}$:
$$S^{\sigma}([x]) = \sigma(x_{1}, S(x_{2}))
\sigma^{-1}(S(x_{4}), x_{5})[ S(x_{3})].$$
The Hopf algebras $Q$ and $Q^{\sigma}$ have equivalent tensor categories of comodules
\cite{sc1}. If $Q$ is a compact Hopf algebra, then $Q^{\sigma}$ is also a compact Hopf algebra, the Haar integral on $Q^\sigma$ being the one of $Q$, and the $\cs^*$-tensor categories of unitary comodules over $Q$ and $Q^\sigma$ are equivalent \cite{bdv}. If $Q= \rep(G)$, the algebra of representative functions on a compact group $G$, we denote by $C(G)^\sigma$ the enveloping $\cs^*$-algebra of $\rep(G)^\sigma$. 

\medskip

Very often unitary 2-cocycles are induced by simpler quotient Hopf $*$-algebras (quantum subgroups).
More precisely let $\pi : Q \to L$ be a Hopf $*$-algebra surjection and let
$\sigma : L \otimes L \to \C$ be a unitary 2-cocycle on $L$.
Then $\sigma_{\pi} = \sigma \circ (\pi \otimes \pi) : Q \otimes Q \to \C$ is a unitary 2-cocycle.
In what follows the cocycle $\sigma_\pi$ will simply be denoted by $\sigma$, this should cause not cause any confusion.

We first record the following elementary result from \cite{bb09}.

\begin{proposition}\label{corespsubgroup}
Let $\pi : Q \to L$ be a Hopf $*$-algebra surjection and let 
 $\sigma : L \otimes L \to \C$ be a unitary $2$-cocycle.  Denote by $[\pi]:Q^\sigma \to L^\sigma$ the map $[x]\mapsto[\pi(x)]$. 
 Then there is a bijection between the following data.
\begin{enumerate}
 \item Surjective Hopf $*$-algebra maps $ f : Q \to R$ such that there exists
 a Hopf $*$-algebra map $g : R \to L$ satisfying $g \circ f = \pi$.
 \item Surjective Hopf $*$-algebra maps $ f' : Q^{\sigma} \to R'$ such that there exists
 a Hopf $*$-algebra map $g' : R' \to L^{\sigma}$ satisfying $g' \circ f' = [\pi]$.
\end{enumerate}
\end{proposition}

Similarly, the following result is essentially contained in \cite{bb09}. 

\begin{proposition}\label{gene}
Let $\pi : Q \to L$ be a Hopf $*$-algebra surjection and let
$\sigma : L \otimes L \to \C$ be a unitary $2$-cocycle on $L$. We have an injective
$*$-algebra map
\begin{align*}
\theta : Q^{\sigma} & \longrightarrow   Q \otimes {_\sigma \! L} \otimes L_{\sigma^{-1}} \\
[x] & \longmapsto x_{2} \otimes \{\pi(x_1)\}  \otimes \langle \pi(x_3) \rangle 
\end{align*}
that induces an isomorphism to the subalgebra of coinvariant elements
$$Q^{\sigma} \simeq   (Q \otimes {_\sigma \! L} \otimes L_{\sigma^{-1}})^{{\rm co}(L^{\rm cop} \otimes L)}$$
where the respective right coactions of $L^{\rm cop} \otimes L$ on $Q$ and ${_\sigma \! L} \otimes L_{\sigma^{-1}}$ are defined by
\begin{align*}
 Q & \rightarrow Q \otimes L^{\rm cop} \otimes L \quad \quad \quad \quad \quad {_\sigma \! L} \otimes L_{\sigma^{-1}} \rightarrow {_\sigma \! L} \otimes L_{\sigma^{-1}}  \otimes L^{\rm cop} \otimes L \\
x & \mapsto x_2 \otimes \pi(x_1) \otimes \pi(x_3) \quad \quad \{\pi(x)\} \otimes \langle \pi(y) \rangle  \mapsto \{\pi(x_1)\} \otimes \langle \pi(y_2)\rangle \otimes S^{-1}\pi(x_2) \otimes S\pi(y_1)
\end{align*}
If moreover $Q$ and $L$ are cosemisimple and $h_Q$ and $h_L$ denote their respective Haar integrals, we have $(h_Q \otimes h_L \otimes h_L)\theta =h_Q$. 
\end{proposition}

\begin{proof}
It follows from the definitions that $\theta$ is a $*$-algebra map and that $({\rm id}_Q \otimes \varepsilon \otimes \varepsilon)\theta = {\rm id}_{Q^\sigma}$, hence $\theta$ is injective. It is a direct verification to check that  
 $\theta(Q^\sigma) \subset (Q \otimes {_\sigma \! L} \otimes L_{\sigma^{-1}})^{{\rm co}(L^{\rm cop} \otimes L)}$, and that $\theta$ induces the announced isomorphism, with inverse $ ({\rm id}_Q \otimes \varepsilon \otimes \varepsilon)$. The last assertion is immediate.
\end{proof}

We now specialize to the case $Q=\rep(G)$, the algebra of representative on a classical compact group $G$.

\begin{proposition}\label{fixcocycle}
 Let $G$ be a compact group, let $\Gamma \subset G$ be a closed subgroup and let $\sigma$ be a unitary $2$-cocycle on $\mathcal R(\Gamma)$. Put $B=  \cs^*_r(_{\sigma}\rep(\Gamma)) \otimes \cs_r^*(\rep(\Gamma)_{\sigma^{-1}})$.  Then there exists a $\cs^*$-algebra embedding
$$\theta : C(G)^\sigma \longrightarrow C(G) \otimes B$$
inducing a $\cs^*$-algebra isomorphism 
$$C(G)^\sigma \simeq (C(G) \otimes B)^{\Gamma^{\rm op} \times \Gamma}$$
for some natural actions of $\Gamma^{\rm op} \times \Gamma$ on $G$ and $B$.
\end{proposition}

\begin{proof}
The restriction map $\mathcal R(G) \rightarrow \mathcal R(\Gamma)$ enables us to use the previous proposition.
The previous injective $*$-algebra map $\theta : \mathcal R(G)^{\sigma} \rightarrow   \rep(G) \otimes {_\sigma \! \rep(\Gamma)} \otimes \rep(\Gamma)_{\sigma^{-1}}$ induces a  $*$-algebra map $C(G)^\sigma \longrightarrow C(G) \otimes B$, still denoted $\theta$ (recall that $C(G)^\sigma$ is the enveloping $\cs^*$-algebra of $\rep(G)^\sigma$). The co-amenability of $\mathcal R(G)^\sigma$ \cite{ba99} and the last observation in the previous proposition show that $\theta$ is injective at the $\cs^*$-algebra level. The coactions of the previous proposition induce actions of $\Gamma^{\rm op} \times \Gamma$ on $\mathcal R(G)$ and on ${_\sigma \! \rep(\Gamma)} \otimes \rep(\Gamma)_{\sigma^{-1}}$, and hence on $C(G)$ and on $B$. We have, by the previous proposition,  an isomorphism $\rep(G)^{\sigma} \simeq (\rep(G) \otimes {_\sigma \! \rep(\Gamma)} \otimes \rep(\Gamma)_{\sigma^{-1}})^{\Gamma^{\rm op} \times \Gamma}$, and hence, since $(\rep(G) \otimes {_\sigma \! \rep(\Gamma)} \otimes \rep(\Gamma)_{\sigma^{-1}})^{\Gamma^{\rm op} \times \Gamma}$ is dense in $(C(G) \otimes B)^{\Gamma^{\rm op} \times \Gamma}$, an isomorphism
$$C(G)^\sigma \simeq (C(G) \otimes B)^{\Gamma^{\rm op} \times \Gamma}$$
This gives the announced result.
\end{proof}

\begin{remark}{\rm 
The right action of $\Gamma^{\rm op} \times \Gamma$ on $G$ in the previous result
is given by
\begin{align*}
 G \times (\Gamma^{\rm op} \times \Gamma) &\longrightarrow G \\
(g, (r,s)) & \longmapsto rgs
\end{align*}
The $C^*$-algebra $(C(G) \otimes B)^{\Gamma^{\rm op} \times \Gamma}$ is naturally identified with
$C(G \times_{\Gamma^{\rm op} \times \Gamma} B)$, the algebra of continuous functions $f : G \rightarrow B$
such that $f(g.(r,s))=(r,s)^{-1}.f(g)$, $\forall g  \in G$, $\forall (r,s) \in \Gamma^{\rm op} \times \Gamma$.
Thus it follows that $C(G)^\sigma$ is (the algebra of sections on) a continuous bundle of $\cs^*$-algebras over the orbit space $G/(\Gamma^{\rm op} \times \Gamma) \simeq \Gamma \setminus G/ \Gamma$, with fiber at an orbit $\Gamma g\Gamma$ the fixed point algebra $B^{(\Gamma^{\rm op} \times \Gamma)_g}$, where $(\Gamma^{\rm op}\times\Gamma)_g= \{(r,s) \in \Gamma \times \Gamma, \ rgs=g\}$: see e.g. Lemma 2.2 in \cite{ee}. Hence the representation theory of $C(G)^{\sigma}$ is determined by the representation theory of the fibres $B^{(\Gamma^{\rm op} \times \Gamma)_g}$.
}
\end{remark}

The following result will be our main tool to study the representations and quotients of a Woronowicz algebra of type $C(G)^\sigma$.

\begin{proposition}\label{maintool}
 Let $G$ be a compact group, let $\Gamma \subset G$ be a closed subgroup and let $\sigma$ be a unitary $2$-cocycle on $\mathcal R(\Gamma)$. Then for each $g \in G$ we have a $*$-algebra map 
\begin{align*}
 \theta_g :  C(G)^{\sigma} & \longrightarrow  \cs_r^*({_\sigma \! \mathcal R(\Gamma})) \otimes \cs_r^*(\mathcal R(\Gamma)_{\sigma^{-1}}) \\
\mathcal R(G)^{\sigma} \ni [f] & \longmapsto f_{2}(g)  \{ {f_1}_{|\Gamma}\}  \otimes \langle {f_3}_{|\Gamma} \rangle \in {_\sigma \! \mathcal R(\Gamma}) \otimes \mathcal R(\Gamma)_{\sigma^{-1}}
\end{align*}
If $\Gamma$ is finite, then  $\dim(\Im(\theta_g))= |\Gamma g \Gamma|$.
 
Assume moreover that
${_\sigma \! \mathcal R(\Gamma})$ and $\mathcal R(\Gamma)_{\sigma^{-1}}$ are full matrix algebras, so that $\theta_g$ defines a representation of dimension $|\Gamma|$ of   $\mathcal R(G)^{\sigma}$. 
\begin{enumerate}
\item Every irreducible representation of $C(G)^\sigma$ is isomorphic to a  subrepresentation of $\theta_g$ for some $g \in G$. In particular every irreducible representation of $C(G)^\sigma$ is finite-dimensional and has dimension at most $|\Gamma|$.
\item  The representation $\theta_g$ is irreducible if and only if $|\Gamma g \Gamma|= |\Gamma|^2$, if and only if $\#\{(s,t) \in \Gamma \times \Gamma \ | \ sgt=g\}=1$. Any irreducible representation of dimension $| \Gamma|$ of $C(G)^{\sigma}$ is isomorphic to an irreducible  representation $\theta_g$ as above.
\item For $g,h \in G$, we have $\theta_g \simeq \theta_h \iff \Gamma g \Gamma = \Gamma h \Gamma$.
\item For $g,h \in G$, we have $\theta_g \otimes \theta_h \simeq \oplus_{s \in \Gamma} \theta_{gsh}$.
\item Assume furthermore that $\Gamma$ is abelian. Then each $s \in \Gamma$ defines a 1-dimensional representation $\varepsilon_s$ of $C(G)^\sigma$, and for $s \in \Gamma$, we have $\theta_s\simeq \oplus_{t\in \Gamma} \varepsilon_t$.  
\end{enumerate}
\end{proposition}

\begin{proof}
The representations $\theta_g$ are defined using the previous embedding $\theta$, by $\theta_g= ({\rm ev}_g \otimes {\rm id} \otimes {\rm id})\theta$, where ${\rm ev}_g$ is the evaluation at $g$. 
We  assume now that $\Gamma$ is finite. As a linear space, we view $\cs_r^*({_\sigma \! \mathcal R(\Gamma})) \otimes \cs_r^*(\mathcal R(\Gamma)_{\sigma^{-1}})$ as $C(\Gamma \times \Gamma)$.
Consider the continuous linear map
\begin{align*}
 \theta'_g : C(G) &\longrightarrow C(\Gamma \times \Gamma) \\
f &\longmapsto ( (s,t) \mapsto f(sgt))
\end{align*}
For $f \in \mathcal R(G)$, we have $\theta'_g(f)=\theta_g([f])$, hence  $\theta'_g(\mathcal R(G))=\theta_g(\mathcal R(G)^\sigma)$ and  $\theta'_g(C(G))=\theta_g(C(G)^\sigma)$ by the density of $\mathcal R(G)$ and the finite-dimensionality of the target space.
We have ${\rm Ker (\theta'_g)} =\{f \in C(G) \ | \ f_{|\Gamma g\Gamma}=0\}=I$ and since $\theta'_g(C(G)) \simeq C(G)/I\simeq C(\Gamma g \Gamma)$, we have $\dim(\theta'_g(C(G))$ = $|\Gamma g \Gamma|=\dim(\theta_g(C(G)^\sigma))$. 

Assume now that
${_\sigma \! \mathcal R(\Gamma})$ and $\mathcal R(\Gamma)_{\sigma^{-1}}$ are full matrix algebras.  
By counting dimensions, ${_\sigma \! \mathcal R(\Gamma}) \otimes \mathcal R(\Gamma)_{\sigma^{-1}} \cong M_{|\Gamma|}(\C)$.
The irreducible representations  of $C(G) \otimes {_\sigma \! \mathcal R(\Gamma}) \otimes \mathcal R(\Gamma)_{\sigma^{-1}}$ all are of the form ${\rm ev}_g \otimes {\rm id} \otimes {\rm id}$, and since 
$\theta$ defines an embedding $C(G)^\sigma \hookrightarrow C(G) \otimes {_\sigma \! \mathcal R(\Gamma}) \otimes \mathcal R(\Gamma)_{\sigma^{-1}}$, it follows from Theorem \ref{extend} that any irreducible representation of $C(G)^\sigma$ is isomorphic to a  subrepresentation of some $\theta_g$, and hence is finite-dimensional of dimension $\leq |\Gamma|$. This proves (1). The matrix representation $\theta_g$ is irreducible if and only if $\theta_g$ is surjective, if and only if $|\Gamma g \Gamma|=\dim({\rm Im}(\theta_g))= |\Gamma|^2$, and this proves (2).

Consider now the linear map
\begin{align}
\label{eq:character}
 \chi'_g : C(G) & \longrightarrow \C \\
f & \mapsto \frac{1}{|\Gamma|} \sum_{s,t \in \Gamma}f(sgt) \nonumber
\end{align}
Let $\chi_g$ be the character of $\theta_g$. Let us check that $\chi_g([f])=\chi'_g(f)$ for any $f \in \mathcal R(G)$. By the density of $\mathcal R(G)$ and $\mathcal R(G)^\sigma$, this will show that for $g,h \in G$, we have $\chi_g=\chi_h \iff \chi'_g=\chi'_h$. 
 Consider the normalized Haar integral $h : C(\Gamma) \rightarrow \C$, $f \mapsto \frac{1}{|\Gamma|} \sum_{s \in \Gamma}f(s)$. Then $h$, viewed as a linear map on 
 ${_\sigma \! \mathcal R(\Gamma})$, is still a trace since it is invariant under the natural ergodic action of the finite group $\Gamma$ on the matrix algebra  ${_\sigma \! \mathcal R(\Gamma})$, and hence we have $h=\frac{1}{\sqrt{|\Gamma|}} {\rm tr}$, where ${\rm tr}$ is the usual trace. Thus we have, for $f \in \mathcal R(G)$,
\begin{align*}
 \chi_g([f])=({\rm tr} \otimes {\rm tr})\theta_g ([f]) &= |\Gamma|(h \otimes h)\theta_g([f])=  |\Gamma|(h\otimes h)( f_{2}(g)  \{f_{1_{|\Gamma}} \} \otimes \langle f_{3_{|\Gamma}} \rangle) \\
& = \frac{1}{|\Gamma|} \sum_{s,t \in \Gamma} f_1(s) f_2(g) f_3(t) = \frac{1}{|\Gamma|} \sum_{s,t \in \Gamma}
f(sgt) = \chi'_g(f)
\end{align*}

Let $g,h \in G$. If $\Gamma g \Gamma=\Gamma h \Gamma$, then $\chi'_g=\chi'_h$, and hence $\chi_g=\chi_h$, and it follows that $\theta_g \simeq \theta_h$.
Conversely, assume that $\Gamma g \Gamma\not=\Gamma h \Gamma$, and let $f \in C(G)$ be such that
$f_{|\Gamma g \Gamma}=0$ and $f_{|\Gamma h \Gamma}=1$. We have $\chi'_g(f)=0$ and $\chi'_h(f)=1$: this shows that $\chi_g \not=\chi_h$ and hence that $\theta_g$ and $\theta_h$ are not isomorphic.
This proves (3).

For $g,h \in G$, let us show that $(\chi_g \otimes \chi_h)\Delta = \sum_{s \in \Gamma} \chi_{gsh}$. This will prove (4). For $f$ in $\mathcal R(G)$, we have 
\begin{align*}(\chi_g \otimes \chi_h)\Delta([f]) &= \chi_g([f_1])\chi_h([f_2])=\frac{1}{|\Gamma|^2}\sum_{r,s,t,u \in \Gamma}f_1(rgs)f_2(tsu) \\
& = \frac{1}{|\Gamma|^2}\sum_{r,s,t,u \in \Gamma} f(rgsthu) =  \frac{1}{|\Gamma|}\sum_{r,s,u \in \Gamma} f(rgshu)
= \sum_{s \in \Gamma} \chi_{gsh}([f])
\end{align*}
and we have the result by density of $\mathcal R(G)^\sigma$ in $C(G)^\sigma$.

Assume finally that $\Gamma$ is abelian. Then $\rep(\Gamma)$ is cocommutative and $\rep(\Gamma)^\sigma=\rep(\Gamma)$. For $s \in \Gamma$, the $*$-algebra map $\varepsilon_s : \rep(G)^\sigma \rightarrow \C$ is obtained by composing the restriction $\mathcal R(G)^\sigma \rightarrow \rep(\Gamma)^\sigma=\rep(\Gamma)$ with the evaluation at $s$. For $s \in \Gamma$ and $f$ in $\rep(G)$, we have
$$\chi_s([f])=\frac{1}{|\Gamma|}\sum_{r,t \in \Gamma} f(rst) = \frac{1}{|\Gamma|}\sum_{r,t \in \Gamma} \varepsilon_{rst}([f])=\sum_{r \in \Gamma}\varepsilon_r([f])$$
and again we get the result  by density of $\mathcal R(G)^\sigma$ in $C(G)^\sigma$.
\end{proof}

We arrive at a useful criterion to show that a quotient of a twisted function algebra on compact group is still a twisted function algebra on a compact subgroup.

\begin{theorem}\label{critsub}
 Let $G$ be a compact group and let $\sigma$ be a unitary 2-cocycle on $\rep(G)$ induced by a finite abelian subgroup $\Gamma \subset G$ such that ${_\sigma \! \mathcal R(\Gamma})$ is a full matrix algebra. Let $A$ be a Woronowicz algebra quotient of $C(G)^\sigma$. Then all the irreducible representations of the $\cs^*$-algebra $A$ have dimension $\leq |\Gamma|$, and if $A$ has an irreducible representation of dimension $|\Gamma|$, then there exists a compact subgroup $\Gamma \subset K \subset G$ such that $A \simeq C(K)^\sigma$.
\end{theorem}

\begin{proof}
We are in the situation of Proposition \ref{maintool}, since the algebras ${_\sigma \! \mathcal R(\Gamma})$ and $\mathcal R(\Gamma)_{\sigma^{-1}}$ are anti-isomorphic.   Thus if $\rho$ is an irreducible representation of $A$ of dimension $|\Gamma|$, then $\rho\pi$ is also an irreducible representation of $C(G)^\sigma$ (with  $\pi : C(G)^\sigma \rightarrow A$ being the given quotient map), and so there exists $g \in G$ such that $\rho\pi\simeq \theta_g$. That is, $\theta_g$ factors through a representation of $A$. The isomorphisms from \ref{maintool}
 $$\theta_g \otimes \theta_{g^{-1}} \simeq \oplus_{s \in \Gamma} \theta_{gsg^{-1}}\simeq \theta_{1} \oplus ( \oplus_{s \in \Gamma, s\not=1} \theta_{gsh})\simeq 
(\oplus_{s \in \Gamma}\varepsilon_s)\oplus (\oplus_{s \in \Gamma, s\not=1} \theta_{gsh}) $$
show that $\theta_{g^{-1}}$ is the dual of the representation $\theta_g$ of $C(G)^{\sigma}$. Thus, $\theta_{g^{-1}}$ factors through a representation of $A$, as do all the simple constituents of  $\theta_g \otimes \theta_{g^{-1}}$. In particular, each $\varepsilon_s$, $s \in \Gamma$, defines a representation $A$, and we get a surjective $*$-algebra map $A \rightarrow \rep(\Gamma)$. We conclude by Proposition \ref{corespsubgroup}.
\end{proof}

\section{Application to $SU_{-1}(2m+1)$ and $U_{-1}(2m+1)$}

From now on we assume that $n=2m+1$ is odd.
We recall how the quantum groups   $SU_{-1}(2m+1)$ and $U_{-1}(2m+1)$ can be obtained by 2-cocycle deformation, using a 2-cocycle induced from the group $\mathbb Z_2^{2m}$, and then use the results of the previous section to get information on their quantum subgroups.

We denote by $\mathbb Z_2$ the cyclic group on two elements, and we use the identification
$$\mathbb Z_2^{2m} = \langle t_1, \ldots , t_{2m+1} \ | \ t_it_j=t_it_j, \ t_1^2=\cdots = t_{2m+1}^2= 1= t_1 \cdots t_{2m+1} \rangle$$
Let $\sigma : \mathbb Z_2^{2m} \times \mathbb Z_2^{2m} \rightarrow \{\pm 1\}$ be the unique bicharacter such that
$$\sigma(t_i,t_j)= 
                    -1  = -\sigma(t_j,t_i) \ \text{for $1\leq i<j \leq 2m$}$$
$$\sigma(t_i,t_i)=(-1)^m  \ \text{for $1 \leq i\leq 2m+1$}$$ 
$$\sigma(t_i,t_{2m+1})= (-1)^{m-i} = -\sigma(t_{2m+1},t_i)  \ \text{ for $1\leq i\leq 2m$}
                   $$
It is well-known that the twisted group algebra $\C_\sigma \mathbb Z_2^{2m}$ is isomorphic to the matrix algebra
$M_{2^{m}}(\C)$.

There exists a surjective Hopf $*$-algebra morphism 
\begin{align*}
 \pi : \mathcal R(SU(2m+1)) &\rightarrow \C \mathbb Z_2^{2m} \\
u_{ij} & \longmapsto \delta_{ij}t_{i}
\end{align*}
induced by the restriction of functions to $\Gamma$, the subgroup of $SU(2m+1)$ formed by diagonal matrices having $\pm 1$ as entries, composed with the Fourier transform $\rep(\Gamma) \simeq \C \widehat{\Gamma} \simeq  \C  \mathbb Z_2^{2m}$. Thus we may form the twisted Hopf algebra $\mathcal R(SU(2m+1))^\sigma$, and it is not difficult to check that there exists a surjective Hopf $*$-algebra map $\mathcal R(SU_{-1}(2m+1)) \rightarrow \mathcal R(SU(2m+1))^\sigma$, $u_{ij} \mapsto [u_{ij}]$, which is known to be an isomorphism (there are several ways to show this, a simple one being to invoke the presentation Theorem 3.5 in \cite{gkm}). Hence we have $C(SU_{-1}(2m+1))  \simeq  C(SU(2m+1))^\sigma$, with $\sigma$ induced from the subgroup $\Gamma \simeq \mathbb Z_2^{2m}$, and we are in the framework of Theorem \ref{critsub}. Similarly $C(U_{-1}(2m+1))  \simeq  C(U(2m+1))^\sigma$.

 If $K$ is a compact subgroup of $SU(2m+1)$ with $\Gamma \subset K$, we denote by $K_{-1}$ the compact quantum group corresponding to the Woronowicz algebra $C(K)^\sigma$. With this language, the following result is an immediate consequence of Theorem \ref{critsub}.

\begin{theorem}\label{subsu(2m+1)}
 Let $G$ be a compact quantum subgroup of $SU_{-1}(2m+1)$. Then all the irreducible representations of the $\cs^*$-algebra $C(G)$  have dimension $\leq 4^m$, and if   $C(G)$ has an irreducible dimension
of dimension $4^{m}$, then there exists a compact subgroup $\Gamma \subset K \subset SU(2m+1)$ such that $G \simeq K_{-1}$.
\end{theorem}

A similar statement holds as well with $SU_{-1}(2m+1)$ replaced by $U_{-1}(2m+1)$. 

\section{Quantum subgroups of $SU_{-1}(3)$}

This section is devoted to the proof of Theorem \ref{subSu3}. We first need  some preliminary results, and  we begin by fixing some notation.

For a permutation $\nu \in S_3$, we put
$$SU(3)^{\nu}= \{ g=(g_{ij}) \in SU(3) \ | \ g_{ij}=0 \ {\rm if} \ \nu(j)\not=i\}$$
and also
$$SU(3)^{\rm \Sigma} = \cup_{\nu \in S_3} SU(3)^{\nu}.$$
For $g \in SU(3)^\Sigma$, we denote by $\nu_g$ the unique element of $S_3$ such that $g \in SU(3)^{\nu_g}$.

The following result is easily verified (and has an obvious generalization for any $n$).

\begin{lemma}\label{1rep}
 Any element $g =(g_{ij}) \in SU(3)^\Sigma$  defines a $*$-algebra map $\varepsilon_g : C(SU_{-1}(3)) \rightarrow \C$ such that $\varepsilon_g(u_{ij})=\epsilon(\nu_g)g_{ij}$ (where $\epsilon(\nu_g)$ is the signature of $\nu_g$). Conversely any $1$-dimensional representation of $C(SU_{-1}(3))$ arises in this way.
\end{lemma}

As is the previous section, the subgroup of $SU(3)$ formed by diagonal matrices having $\pm 1$ as entries
is denoted $\Gamma$.
In the case $g\in\Gamma$, then $\varepsilon_g$ is of course the representation of the same name from Proposition \ref{maintool}.

We denote by
$SU(3)_{\rm reg}$  the subset of matrices in $SU(3)$ for which there exists a row or a column having no zero coefficient. 

 Recall from Section 4 and Proposition \ref{maintool} that each $g \in SU(3)$ defines a representation
$$\theta_g : C(SU_{-1}(3)) \longrightarrow \C_\sigma\Gamma \otimes \C_\sigma\Gamma \simeq M_2(\C) \otimes M_2(\C) \simeq M_4(\C)$$
The twisted group algebra $\C_\sigma\Gamma$ is presented by generators $T_1$, $T_2$, $T_3$ and relations $T_1^2=-1=T_2^2=T_3^2$, $1=T_{1}T_{2}T_3$, $T_iT_j=-T_jT_i$ if $i \not=j$
(where in the notation of the previous sections, $T_i= \{t_i\}=\langle t_i\rangle$).
With this notation, the representation $\theta_g$ ($g \in SU(3)$) has the following form
\begin{align*}
 \theta_{g} : C(SU_{-1}(3)) &\longrightarrow \C_\sigma\Gamma \otimes \C_\sigma\Gamma \\
u_{ij} &\longmapsto g_{ij} T_i \otimes T_j
\end{align*}

\begin{lemma}\label{critirred}
 The representation $\theta_g$ is irreducible if and only if $g \in SU(3)_{\rm reg}$. If $g \in SU(3)^{\Sigma}$, then $\theta_g$ is isomorphic to a direct sum of one-dimensional representations.
\end{lemma}

\begin{proof}
 The first assertion follows directly from (2) in Proposition \ref{maintool}. The second assertion follows from the fact that if $g \in SU(3)^{\Sigma}$, the algebra $\theta_g(C(SU_{-1}(3))$ is commutative (this is clear from the above description of $\theta_g$).
\end{proof}

Our next aim is to describe the tensor products $\varepsilon_g \otimes \theta_h$.

\begin{lemma}\label{fusion}
 Let $g \in SU(3)^\Sigma$ and let $h \in SU(3)$.  Then the representations $\varepsilon_g \otimes \theta_h$ and $ \theta_{g h}$ are isomorphic.
\end{lemma}

\begin{proof}
Put  $g =(\delta_{i, \nu(j)}a_i)$ with $\nu \in S_3$. We have, for any $i,j$, 
$$(\varepsilon_g \otimes \theta_{h})\Delta(u_{ij}) = \sum_k\varepsilon_g(u_{ik})h_{kj} T_k \otimes T_j
=\epsilon(\nu)a_ih_{\nu^{-1}(i)j} T_{\nu^{-1}(i)}\otimes T_j$$
It is straightforward to check that there exists an automorphism
$\alpha_\nu$ of $\mathbb C_\sigma \Gamma$ such that $\alpha_\nu(T_i)=\varepsilon(\nu)T_{\nu(i)}$ for any $i$. We have 
$$\alpha_\nu \otimes {\rm id}(\varepsilon_g \otimes \theta_{h})\Delta(u_{ij}) 
=a_ih_{\nu^{-1}(i)j}  T_{i}\otimes T_j=\theta_{g h}(u_{ij})$$
and hence, since $\alpha_\nu$ is (necessarily) an inner automorphism of the matrix algebra $\C_\sigma\Gamma$, we conclude that the representations $\varepsilon_g \otimes \theta_h$ and $ \theta_{g h}$ are isomorphic.
\end{proof}

Before going into the proof of Theorem \ref{subSu3}, we need a final piece of notation. For $1\leq i,j \leq 3$, we put 
$$SU(3)^{[i,j]}= \{ g=(g_{ij}) \in SU(3) \ | \ g_{ik}=0 \ {\rm if} \ k\not=j, \ g_{kj}=0 \ {\rm if } \ i\not=k, \ g \not \in SU(3)^{\Sigma}\}$$ 

\begin{proof}[Proof of Theorem \ref{subSu3}]
 Let $G \subset SU_{-1}(3)$ be a non-classical compact quantum subgroup, with corresponding surjective Woronowicz algebra map $\pi : C(SU_{-1}(3)) \rightarrow C(G)$. Recall that we have to prove that one of the following assertion holds.
\begin{enumerate}
 \item There exists a compact subgroup $\Gamma \subset K \subset SU(3)$ such that $G$ is isomorphic to $K_{-1}$.
\item $G$ is isomorphic to a quantum subgroup of $U_{-1}(2)$. 
\end{enumerate}
We already know from Theorem \ref{subsu(2m+1)} that if $C(G)$ has an irreducible representation of dimension 4, then (1) holds. So we assume that $C(G)$ has all its irreducible representation of dimension $<4$.

We denote by $X$ the set of (isomorphism classes) of irreducible representations of $C(G)$ having dimension $d$ satisfying $1<d<4$.  We remark that $X$ is non-empty since $C(G)$ is non-commutative.

Let $\rho \in X$. Then $\rho$ defines an irreducible representation $\rho\pi$ of $C(SU_{-1}(3))$, and hence by Proposition \ref{maintool} there exists $g \in SU(3)$ such that $\rho\pi \prec \theta_g$. If $g \in SU(3)_{\rm reg}$, then by Lemma \ref{critirred} $\theta_g$ is irreducible and $\rho\pi\simeq \theta_g$ has dimension 4, which contradicts our assumptions. Hence
$g \not \in SU(3)_{\rm reg}$. If $g  \in SU(3)^{\Sigma}$, then by Lemma \ref{critirred} $\theta_g$ is a direct sum of representations of dimension $1$, hence $\rho$ has dimension $1$, which again contradicts our assumption, and hence  $g \not \in SU(3)^{\Sigma}$. Thus there exist $i,j$ such that $g \in SU(3)^{[i,j]}$. Suppose that $i \not=j$. Then 
$\rho\pi \otimes \rho\pi \prec \theta_{g} \otimes \theta_g \simeq \oplus_{s \in \Gamma}\theta_{gsg}$ (by Proposition \ref{maintool}). For any $s \in \Gamma$, $sg \in SU(3)^{[i,j]}$ and it is a direct computation to check that $gsg \in SU(3)_{\rm reg}$, so the constituents of this decomposition are irreducible representations. By a dimension argument there exists 
$s \in \Gamma$ such that $\rho\pi \otimes \rho\pi \simeq \theta_{gsg}$, and hence $\rho \otimes \rho$ is irreducible of dimension $4$; this is a contradiction.

We have thus proved that for any $\rho \in X$, there exists $i \in \{1,2,3\}$ and $g  \in SU(3)^{[i,i]}$ such that 
$\rho\pi \prec \theta_g$. Assume that there exist $\rho,\rho' \in X$ with $\rho\pi \prec \theta_g$, $\rho'\pi \prec \theta_{g'}$ for $g \in SU(3)^{[i,i]}$, $g'\in SU(3)^{[j,j]}$ and $i \not=j$. Then $\rho\pi \otimes \rho'\pi \prec \theta_{g} \otimes \theta_{g'} \simeq \oplus_{s \in \Gamma}\theta_{gsg'}$. Once again, for any $s \in \Gamma$, $gsg' \in SU(3)_{\rm reg}$, and we conclude as before that $\rho \otimes \rho'$ is an irreducible representation of dimension $4$, a contradiction.

Thus we have proved that there exists $i \in \{1,2,3 \}$ such that for any $\rho \in X$, we have $\rho\pi \prec \theta_g$ for some $g \in SU(3)^{[i,i]}$, and hence $\rho\pi(u_{ik})=0=\rho\pi(u_{ki})$ for any $k \not=i$ and $\rho \in X$.

Let $\phi$ be a $1$-dimensional representation of $C(G)$. By Lemma \ref{1rep}, there exists $\nu \in S_3$ and $g \in SU(3)^\nu$  such that $\phi\pi= \varepsilon_g$.  Let $\rho \in X$ with $\rho\pi \prec \theta_h$ for $h \in SU(3)^{[i,i]}$. Then $\phi\pi \otimes \rho\pi \prec \varepsilon_g \otimes \theta_h\simeq \theta_{g h}$ by   Lemma \ref{fusion}. It is straightforward to check that $g h \in SU(3)^{[\nu(i),i]}$. By a previous case we must have $\nu(i)=i$. Hence  $\phi\pi(u_{ik})=0=\phi\pi(u_{ki})$ for any $k \not=i$. 

Summarizing, we have shown that for any $\rho \in \widehat{C(G)}$, we have  $\rho\pi(u_{ik})=0=\rho\pi(u_{ki})$ for any $k \not=i$. The irreducible representations of a $\cs^*$-algebra separate its elements, so we conclude that  $\pi(u_{ik})=0=\pi(u_{ki})$ for any $k \not=i$, and by Lemma \ref{reduc}, we are in situation (2). This concludes the proof.
\end{proof}

\begin{corollary}\label{irred}
 Let $G$ be a non-classical compact quantum subgroup of $SU_{-1}(3)$ acting irreducibly on $\C^3$.
Then  $G$ is isomorphic to a $K_{-1}$, a twist at $-1$ of a compact subgroup $K \subset SU(3)$ containing the subgroup of diagonal matrices having $\pm 1$ as entries, and acting irreducibly on $\C^3$.
\end{corollary}

\begin{proof}
We have shown in the previous proof that if $C(G)$ does not have an irreducible representation of dimension $4$, then the fundamental $3$-dimensional representation of $G$ is not irreducible. Thus if $G$ acts irreducibly on $\C^3$, there exist an irreducible representation of dimension 4 of $C(G)$ and    a compact subgroup $\Gamma \subset K \subset SU(3)$ such that $G$ is isomorphic to $K_{-1}$, and $K$ acts irreducibly on $\C^3$ since $G$ does.
\end{proof}

\begin{remark}{\rm 
 The proof of Theorem 1.1 works as well by replacing $SU(3)$ by $SO(3)$. In particular one  recovers, under a less precise form, the results of \cite{bb09}: if  $G \subset SO_{-1}(3)$ is a non-classical compact quantum subgroup, then either there exists a compact subgroup $\Gamma \subset K \subset SO(3)$ such that $G$ is isomorphic to $K_{-1}$ or
$G$ is isomorphic to a quantum subgroup of $O_{-1}(2)$. }
\end{remark}

\begin{remark}{\rm
 Corollary \ref{irred} also holds with $SU_{-1}(3)$ replaced by $U_{-1}(3)$ (and $SU(3)$ replaced by $U(3)$), with a similar proof.
}
\end{remark}

\end{document}